\documentclass[10pt,leqno]{amsart}
\usepackage{indentfirst, amssymb,amsmath,amsthm}
\usepackage{csquotes}
\newtheorem*{theoA}{Theorem A}
\newtheorem*{theoB}{Theorem B}
\newtheorem*{theoC}{Theorem C}
\newtheorem*{theoD}{Theorem D}
\newtheorem*{theoE}{Theorem E}
\newtheorem*{theoF}{Theorem F}
\newtheorem*{theoG}{Theorem G}

\newtheorem{theo}{Theorem}[section]
\newtheorem{lem}{Lemma}[section]

\newtheorem{exm}{Example}[section]
\newtheorem{defi}{Definition}[section]

\newtheorem{question}{Question}[section]
\newcommand{\ol}{\overline}
\newcommand{\be}{\begin{equation}}
\newcommand{\ee}{\end{equation}}
\newcommand{\beas}{\begin{eqnarray*}}
\newcommand{\eeas}{\end{eqnarray*}}
\newcommand{\bea}{\begin{eqnarray}}
\newcommand{\eea}{\end{eqnarray}}
\newcommand{\lra}{\longrightarrow}
\numberwithin{equation}{section}
\AtBeginDocument{{\noindent\small Ann. Univ. Paedagog. Crac. Stud. Math. 14 (2015), 105-119\\
DOI: 10.1515/aupcsm-2015-0008}}
\begin{document}
\title[Further investigations on a question of Zhang and L\"{u} ]{Further investigations on a question of Zhang and L\"{u}}
\date{}
\author[A. Banerjee and B. Chakraborty ]{ Abhijit Banerjee  and Bikash Chakraborty. }
\date{}
\address{ Department of Mathematics, University of Kalyani, West Bengal 741235, India.}
\email{abanerjee\_kal@yahoo.co.in, abanerjee\_kal@rediffmail.com
}
\email{bikashchakraborty.math@yahoo.com, bchakraborty@klyuniv.ac.in}
\maketitle
\let\thefootnote\relax
\footnotetext{Type set by \AmS -\LaTeX}
\setcounter{footnote}{0}

\begin{abstract}
In the paper taking the question of Zhang and L\"{u} \cite{11} into background, we present one theorem which will improve and extend results of Banerjee-Majumder \cite{1} and a recent result of Li-Huang \cite{5a}.  \end{abstract}
\section{Introduction Definitions and Results}
Let $f$ be a non-constant meromorphic function defined in the open complex plane $\mathbb{C}$. We adopt the standard notations of the Nevanlinna theory of meromorphic functions as explained in \cite{4}. \par
If for some $a\in\mathbb{C}\cup\{\infty\}$, $f$ and $g$ have the same set of $a$-points with the same multiplicities, we say that $f$ and $g$ share the value $a$ CM (counting multiplicities) and if we do not consider the multiplicities then $f$, $g$ are said to share the value $a$ IM (ignoring multiplicities).
 When $a=\infty$ the zeros of $f-a$ means the poles of $f$.\par
It will be convenient to let $E$ denote any set of positive real numbers of finite linear measure, not necessarily the same at each occurrence.
For any non-constant meromorphic function $f$, we denote by $S(r, f)$ any quantity satisfying $$S(r, f) = o(T(r, f))\;\;\;\;\;\;\;\;\;\;\ (r\lra \infty, r\not\in E).$$

A meromorphic function $a=a(z)(\not\equiv \infty)$ is called a small function with respect to $f$ provided that $T(r,a)=S(r,f)$ as $r\lra \infty, r\not\in E$.
If $a=a(z)$ is a small function we define that $f$ and $g$ share $a$ IM or $a$ CM according as $f-a$ and $g-a$ share $0$ IM or $0$ CM respectively.
We use $I$ to denote any set of infinite linear measure of $0<r<\infty$.\\
Also it is known to us that the hyper order of $f$, denoted by $\rho _{2}(f)$, is defined by \beas \rho _{2}(f)=\limsup\limits_{r\lra\infty}\frac{ \log \log T(r,f)}{\log r}.\eeas
\par The subject on sharing values between entire functions and their derivatives was first studied by Rubel and Yang (\cite{8}).
\par In 1977, they proved that if a non-constant entire function $f$ and $f^{'}$ share two distinct finite numbers $a$, $b$ CM, then $f = f^{'}$.\par
In 1979, analogous result for IM sharing was obtained by Mues and Steinmetz in the following manner.
\begin{theoA}(\cite{7}) Let $f$ be a non-constant entire function. If $f$ and $f^{'}$ share two distinct values $a$, $b$ IM then $f^{'}\equiv f$.\end{theoA}
Subsequently, similar considerations have been made with respect to higher derivatives and more general differential expressions as well.

Above theorems motivate the researchers to study the relation between an entire function and its derivative counterpart for one CM shared value. In 1996, in this direction the following famous conjecture was proposed by Br\"{u}ck (\cite{3}) :\\

{\it {\bf Conjecture :} Let $f$ be a non-constant entire function such that the hyper order $\rho _{2}(f)$ of $f$ is not a positive integer or infinite. If $f$ and $f^{'}$ share a finite value $a$ CM, then $\frac{f^{'}-a}{f-a}=c$, where $c$ is a non-zero constant.} \par
Br\"{u}ck himself proved the conjecture for $a=0$. For $a\not=0$, Br\"{u}ck (\cite{3}) obtained the following result in which additional supposition was required.
\begin{theoB}(\cite{3}) Let $f$ be a non-constant entire function. If $f$ and $f^{'}$ share the value $1$ CM and if $N(r,0;f^{'})=S(r,f)$ then $\frac{f^{'}-1}{f-1}$ is a nonzero constant.\end{theoB}

Next we recall the following definitions.

\begin{defi} (\cite{5}) Let $p$ be a positive integer and $a\in\mathbb{C}\cup\{\infty\}$.\begin{enumerate}
\item[(i)] $N(r,a;f\mid \geq p)$ ($\ol N(r,a;f\mid \geq p)$)denotes the counting function (reduced counting function) of those $a$-points of $f$ whose multiplicities are not less than $p$.\item[(ii)]$N(r,a;f\mid \leq p)$ ($\ol N(r,a;f\mid \leq p)$)denotes the counting function (reduced counting function) of those $a$-points of $f$ whose multiplicities are not greater than $p$.\end{enumerate} \end{defi}
\begin{defi} (\cite{10}) For $a\in\mathbb{C}\cup\{\infty\}$ and a positive integer $p$ we denote by $N_{p}(r,a;f)$ the sum $\ol N(r,a;f)+\ol N(r,a;f\mid \geq 2)+\ldots+\ol N(r,a;f\mid \geq p)$. Clearly $N_{1}(r,a;f)=\ol N(r,a;f)$. \end{defi}

\begin{defi} (\cite{10}) For $a \in \mathbb{C}\cup\{\infty\}$ and a positive integer p we put\\
$$\delta_{p}(a,f)=1-\limsup\limits_{r \to \infty} \frac{{N}_{p}(r,a;f)}{T(r,f)}.$$\\ Clearly $0\leq \delta(a,f) \leq \delta_{p}(a,f) \leq \delta_{p-1}(a,f) \leq...\leq \delta_{2}(a,f) \leq \delta_{1}(a,f)=\Theta(a,f)\leq1$ .\end{defi}
\begin{defi} For two positive integers $n$, $p$ we define \\
$\mu_{p}= min\{n,p\}$ and $\mu_{p}^{*}= p+1-\mu_{p}$. Then clearly $$ N_{p}(r,0;f^{n}) \leq \mu_{p}N_{\mu_{p}^{*}}(r,0;f).$$
\end{defi}
\begin{defi} (\cite{1}) Let $z_{0}$ be a zero of $f-a$ of multiplicity $p$ and a zero of $g-a$ of multiplicity $q$. We denote by $\ol N_{L}(r,a;f)$ the counting function of those $a$-points of $f$ and $g$ where $p>q\geq 1$, by $N^{1)}_{E}(r,a;f)$ the counting function of those $a$-points of $f$ and $g$ where $p=q=1$ and by $\ol N^{(2}_{E}(r,a;f)$ the counting function of those $a$-points of $f$ and $g$ where $p=q\geq 2$, each point in these counting functions is counted only once. In the same way we can define $\ol N_{L}(r,a;g),\; N^{1)}_{E}(r,a;g),\; \ol N^{(2}_{E}(r,a;g).$\end{defi}

\begin{defi} (\cite{4a}) Let $k$ be a nonnegative integer or infinity. For $a\in\mathbb{C}\cup\{\infty\}$ we denote by $E_{k}(a;f)$ the set of all $a$-points of $f$, where an $a$-point of multiplicity $m$ is counted $m$ times if $m\leq k$ and $k+1$ times if $m>k$. If $E_{k}(a;f)=E_{k}(a;g)$, we say that $f,g$ share the value $a$ with weight $k$.\end{defi}
The definition implies that if $f$, $g$ share a value $a$ with weight $k$ then $z_{0}$ is an $a$-point of $f$ with multiplicity $m\;(\leq k)$ if and only if it is an $a$-point of $g$ with multiplicity $m\;(\leq k)$ and $z_{0}$ is an $a$-point of $f$ with multiplicity $m\;(>k)$ if and only if it is an $a$-point of $g$ with multiplicity $n\;(>k)$, where $m$ is not necessarily equal to $n$.

We write $f$, $g$ share $(a,k)$ to mean that $f$, $g$ share the value $a$ with weight $k$. Clearly if $f$, $g$ share $(a,k)$, then $f$, $g$ share $(a,p)$ for any integer $p$, $0\leq p<k$. Also we note that $f$, $g$ share a value $a$ IM or CM if and only if $f$, $g$ share $(a,0)$ or $(a,\infty)$ respectively.\par
With the notion of weighted sharing of values Lahiri-Sarkar (\cite{5}) improved the result of Zhang (\cite{9}). In (\cite{10}) Zhang extended the result of Lahiri-Sarkar (\cite{5}) and replaced the concept of value sharing by small function sharing.\par
In 2008 Zhang and L\"{u}(\cite{11}) considered the uniqueness of the $n-$th power of a meromorphic function sharing a small function with its $k-$ th derivative and proved  the following theorem.
\begin{theoC}(\cite{11}) Let $ k(\geq 1)$, $n(\geq 1)$ be integers and $f$ be a non-constant meromorphic function. Also let $a(z) (\not\equiv 0,\infty )$ be a small function with respect to $f$. Suppose $f^{n}-a$ and $f^{(k)}-a$ share $(0,l)$. If $l=\infty$ and \par
\be \label {e1.1}(3+k)\Theta(\infty,f)+2\Theta(0,f)+\delta_{2+k}(0,f) > 6+k-n \ee\\
or $l=0$ and \par
\be \label {e1.2}(6+2k)\Theta(\infty,f)+4\Theta(0,f)+2\delta_{2+k}(0,f) > 12+2k-n \ee\\
then $f^{n}$ $\equiv$ $f^{(k)}$ .\end{theoC}

In the same paper Zhang and L\"{u} (\cite{11}) raised the following question :\\
What will happen if  $f^{n}$ and $[f^{(k)}]^{s}$ share a small function ?\par
In  2010, Chen and Zhang (\cite{2}) gave a answer to the above question. Unfortunately there were some gaps in the proof of the theorems in (\cite{2}) which was latter rectified by Banerjee and Majumder (\cite{1}).
\par In  2010 Banerjee and Majumder (\cite{1}) proved two theorems one of which further improved {\it Theorem C} whereas the other answers the open question of Zhang and L\"{u} (\cite{11}) in the following manner.
\begin{theoD} (\cite{1}) Let $ k(\geq 1)$, $n(\geq 1)$ be integers and $f$ be a non-constant meromorphic function. Also let $a(z) (\not\equiv 0,\infty )$ be a small function with respect to $f$. Suppose $f^{n}-a$ and $f^{(k)}-a$ share $(0,l)$. If $l\geq 2$ and \par
\be \label {e1.3} (3+k)\Theta(\infty,f)+2\Theta(0,f)+\delta_{2+k}(0,f) > 6+k-n \ee\\
or $l=1$ and \par
\be \label {e1.4} \left(\frac{7}{2}+k\right)\Theta(\infty,f)+\frac{5}{2}\Theta(0,f)+\delta_{2+k}(0,f) > 7+k-n \ee\\
or $l=0$ and \par
\be \label {e1.5}(6+2k)\Theta(\infty,f)+4\Theta(0,f)+\delta_{2+k}(0,f)+\delta_{1+k}(0,f) > 12+2k-n \ee\\
then $f^{n}=f^{(k)}$ .\end{theoD}
\begin{theoE} (\cite{1}) Let $k(\geq 1)$, $n(\geq 1)$, $m(\geq 2)$ be integers and $f$ be a non-constant meromorphic function.
Also let $a(z)(\not\equiv 0,\infty)$ be a small function with respect to $f$. Suppose $f^{n}-a$ and $[f^{(k)}]^{m}-a$ share $(0,l)$.
If $l=2$ and \be\label{e1.3a}(3+2k)\;\Theta (\infty,f)+2\;\Theta (0,f)+2\delta_{1+k}(0,f)> 7+2k-n\ee or
 $l=1$ and \be\label{e1.4a}\left(\frac{7}{2}+2k\right)\;\Theta (\infty,f)+\frac{5}{2}\;\Theta (0,f)+2\delta_{1+k}(0,f)> 8+2k-n\ee
or $l=0$ and \be\label{e1.5a}(6+3k)\;\Theta (\infty,f)+4\;\Theta (0,f)+3\delta_{1+k}(0,f)> 13+3k-n\ee then $f^{n}\equiv [f^{(k)}]^{m}$.\end{theoE}
For $m=1$ it can be easily proved that {\it Theorem D} is a better result than {\it Theorem E}. Also we observe that in the conditions (\ref{e1.3a})-(\ref{e1.5a}) there was no influence of $m$. \par
Very recently, in order to improve the results of Zhang (\cite{10}), Li-Huang (\cite{5a}) obtained the following theorem. In view of {\it Lemma \ref{l1.1}} proved latter on, we see that the following result obtained in (\cite{5a}) is better than that of {\it Theorem D} for $n=1$.
\begin{theoF}(\cite{5a}) Let $f$ be a non-constant meromorphic function, $ k(\geq 1)$, $l(\geq 0)$ be be integers and also let $a(z) (\not\equiv 0,\infty )$ be a small function with respect to $f$. Suppose $f-a$ and $f^{(k)}-a$ share $(0,l)$. If $l\geq 2$ and \par
\be \label {e1.6} (3+k)\Theta(\infty,f)+\delta_{2}(0,f)+\delta_{2+k}(0,f) > k+4\ee\\
or $l=1$ and \par
\be \label {e1.7} \left(\frac{7}{2}+k\right)\Theta(\infty,f)+\frac{1}{2}\Theta(0,f)+\delta_{2}(0,f)+\delta_{2+k}(0,f) > k+5 \ee\\
or $l=0$ and \par
\be \label {e1.8} (6+2k)\Theta(\infty,f)+2\Theta(0,f)+\delta_{2}(0,f)+\delta_{1+k}(0,f)+\delta_{2+k}(0,f) > 2k+10 \ee\\
then $f \equiv f^{(k)}$.\end{theoF}
Next we recall the following definition.
\begin{defi} (\cite{4}) Let $n_{0j},n_{1j},\ldots,n_{kj}$ be nonnegative integers.\\
The expression $M_{j}[f]=(f)^{n_{0j}}(f^{(1)})^{n_{1j}}\ldots(f^{(k)})^{n_{kj}}$ is called a differential monomial generated by $f$ of degree $d_{M_{j}}=d(M_{j})=\sum\limits_{i=0}^{k}n_{ij}$ and weight
$\Gamma_{M_{j}}=\sum\limits_{i=0}^{k}(i+1)n_{ij}$.

The sum $P[f]=\sum\limits_{j=1}^{t}b_{j}M_{j}[f]$ is called a differential polynomial generated by $f$ of degree $\ol{d}(P)=max\{d(M_{j}):1\leq j\leq t\}$
and weight $\Gamma_{P}=max\{\Gamma_{M_{j}}:1\leq j\leq t\}$, where $T(r,b_{j})=S(r,f)$ for $j=1,2,\ldots,t$.

The numbers $\underline{d}(P)=min\{d(M_{j}):1\leq j\leq t\}$ and k (the highest order of the derivative of $f$ in $P[f]$) are called respectively the lower degree and order of $P[f]$.

$P[f]$ is said to be homogeneous if $\ol{d}(P)$=$\underline{d}(P)$.

$P[f]$ is called a linear differential polynomial generated by $f$ if $\ol {d} (P)=1$. Otherwise $P[f]$ is called a non-linear differential polynomial.\par
We denote by $Q=max\; \{\Gamma _{M_{j}}-d(M_{j}): 1\leq j\leq t\}=max\; \{ n_{1j}+2n_{2j}+\ldots+kn_{kj}: 1\leq j\leq t\}$.\par
Also for the sake of convenience for a differential monomial $M[f]$ we denote by  $\lambda =\Gamma_{M}-d_{M}$.
\end{defi}
Recently Charak-Lal (\cite{2a}) considered the possible extension of {\it Theorem D} in the direction of the question of Zhang and L\"{u} (\cite{11}) up to differential polynomial.\par They proved the following result :
\begin{theoG} (\cite{2a}) Let $f$ be a non-constant meromorphic function and $n$  be a positive integer and $a(z) (\not\equiv 0,\infty )$ be a meromorphic function satisfying $T(r,a)=o(T(r,f))$ as $r \to \infty$. Let $P[f]$ be a non-constant differential polynomial in $f$. Suppose $f^{n}$ and $P[f]$ share $(a,l)$. If $l\geq 2$ and \par
\be \label {e1.9}(3+Q)\Theta(\infty,f)+2\Theta(0,f)+\ol{d}(P)\delta(0,f) > Q+5+2\ol{d}(P)-\underline{d}(P)-n \ee\\
or $l=1$ and \par
\be \label {e1.10} \left(\frac{7}{2}+Q\right)\Theta(\infty,f)+\frac{5}{2}\Theta(0,f)+\ol{d}(P)\delta(0,f) > Q+6+2\ol{d}(P)-\underline{d}(P)-n \ee\\
or $l=0$ and \par
\be \label {e1.11} (6+2Q)\Theta(\infty,f)+4\Theta(0,f)+2\ol{d}(P)\delta(0,f) > 2Q+4\ol{d}(P)-2\underline{d}(P)+10-n \ee\\
then $f^{n} \equiv P[f]$.\end{theoG}
This is a supplementary result corresponding to {\it Theorem D} because putting $P[f]=f^{(k)}$ one can't obtain {\it Theorem D}, rather in this case a set of stronger conditions are obtained  as particular case of {\it Theorem F}. So it is natural to ask the next question.
\begin {question} \label{q1}
\emph{Is it possible to improve {\it Theorem D} in the direction of {\it Theorem F} up to differential monomial so that the result give a positive answer to the question of Zhang and L\"{u} \cite{11} ?}\end{question}
To seek the possible answer of  {\it Question 1.1} is the motivation of the paper.\\
The following theorem is the main result of this paper which gives a positive answer of Zhang and L\"{u}(\cite{11}).
\begin{theo}\label{t1} Let $ k(\geq 1)$, $n(\geq 1)$ be integers and $f$ be a non-constant meromorphic function and $M[f]$ be a differential monomial of degree $d_{M}$ and weight $\Gamma_{M}$ and $k$ is the highest derivative in $M[f]$. Also let $a(z) (\not\equiv 0,\infty )$ be a small function with respect to $f$. Suppose $f^{n}-a$ and $M[f]-a$ share $(0,l)$. If $l\geq 2$ and \par
\be \label{e1.12} (3+\lambda)\Theta(\infty,f)+\mu_{2}\delta_{\mu_{2}^{*}}(0,f)+d_{M}\delta_{2+k}(0,f) > 3+\Gamma_{M}+\mu_{2}-n \ee\\
or $l=1$ and \par
\be\label {e1.13}  \left(\frac{7}{2}+\lambda\right)\Theta(\infty,f)+\frac{1}{2}\Theta(0,f)+\mu_{2}\delta_{\mu_{2}^{*}}(0,f)+d_{M}\delta_{2+k}(0,f) > 4+\Gamma_{M}+\mu_{2}-n \ee\\
or $l=0$ and \par
\be \label {e1.14} (6+2\lambda)\Theta(\infty,f)+2\Theta(0,f)+\mu_{2}\delta_{\mu_{2}^{*}}(0,f)+d_{M}\delta_{2+k}(0,f)+d_{M}\delta_{1+k}(0,f) > 8+2\Gamma_{M}+\mu_{2}-n \ee\\
then $f^{n} \equiv M[f]$ .
\end{theo}
However the following question is still open.
\begin{question}
Is it possible to extend Theorem \ref{t1} up to differential polynomial instead of differential monomial ?
\end{question}
Following example shows that in Theorem \ref{t1} $a(z) \not\equiv 0,\infty $ is necessary.
\begin{exm} \par
Let us take $f(z)=e^{e^{z}}$ and $M=f'$ then $M$ and $f$ share $0$ (or, $\infty$) and the deficiency conditions stated in theorem \ref{t1} is satisfied as $0$, $\infty$ both are exceptional values of f but $f \not\equiv M$.
\end{exm}
The next example shows that the deficiency conditions stated in Theorem \ref{t1} are not necessary.
\begin{exm} \par
Let $f(z)=Ae^{z}+Be^{-z}$, $AB \neq 0$. Then $\ol{N}(r,f)=S(r,f)$ and $\ol{N}(r,0;f)=\ol{N}(r,-\frac{B}{A};e^{2z})\sim T(r,f)$. Thus $\Theta(\infty,f)=1$ and $\Theta(0,f)=\delta_{p}(0,f)=0$. \par
It is clear that $M[f]=f^{''}$ and $f$ share $a(z)=\frac{1}{z}$ and the deficiency conditions in theorem \ref{t1} is not satisfied, but $M \equiv f$.
\end{exm}
In the next example we see that $f^{n}$ can't be replaced by arbitrary polynomial $P[f]=a_{0}f^{n}+a_{1}f^{n-1}+\ldots+a_{n}$ in Theorem \ref{t1} for IM sharing ($l=0$) case.
\begin{exm} \label{ex1.3} \par
If we take $f(z)=e^{z}$, $P[f]=f^{2}+2f$ and $M[f]=f^{(3)}$, then $P+1=(M+1)^{2}$. Thus $P$ and $M$ share $(-1,0)$. Also $\Theta(0,f)=\Theta(\infty,f)=\delta_{p}(0,f)=\delta(0,f)=1$ as $0$ and $\infty$ are exceptional values of $f$. Thus (\ref{e1.14}) of theorem \ref{t1} is satisfied but $P\not\equiv M$.
\end{exm}
In view of example \ref{ex1.3} the following question is inevitable.
\begin{question}
Is it possible to replace $f^{n}$ by arbitrary polynomial $P[f]=a_{0}f^{n}+a_{1}f^{n-1}+...+a_{n}$ in Theorem \ref{t1} for $l\geq1$ ?
\end{question}

\section{Lemmas} In this section we present some Lemmas which will be needed in the sequel. Let $F$, $G$ be two non-constant meromorphic functions. Henceforth we shall denote by $H$ the following function. \be\label{e2.1}H=\left(\frac{\;\;F^{''}}{F^{'}}-\frac{2F^{'}}{F-1}\right)-\left(\frac{\;\;G^{''}}{G^{'}}-\frac{2G^{'}}{G-1}\right).\ee
\begin{lem}\label{l1.1} $1+\delta_{2}(0,f) \geq 2\Theta(0,f)$.\end{lem}
\begin{proof} \beas 2\Theta(0,f)-\delta_{2}(0,f)-1 &=& \limsup_{r \to \infty}\frac{N_{2}(r,0;f)}{T(r,f)}-\limsup_{r \to \infty}\frac{2\ol{N}(r,0;f)}{T(r,f)}\\
& \leq& \limsup_{r \to \infty}\frac{N_{2}(r,0;f)-2\overline{N}(r,0;f)}{T(r,f)}\\
& \leq& 0. \eeas
\end{proof}
The following three Lemmas can be proved using Milloux Theorem (\cite{4}). So we omit the details.
\begin{lem}\label{l1}  Let $f$ be a non-constant  meromorphic function and $M[f]$  be a differential monomial of degree $d_{M}$ and weight $\Gamma_{M}$. Then $T(r,M)$ $\leq$ $d_{M}T(r,f) + \lambda\overline{N}(r,\infty;f) +S(r,f)$.\end{lem}

\begin{lem}\label{l2}
$N(r,0;M) \leq T(r,M)-d_{M}T(r,f)+d_{M}N(r,0;f)+S(r,f).$
\end{lem}

\begin{lem}\label{l3} $N(r,0;M) \leq d_{M}N(r,0;f)+\lambda\overline{N}(r,\infty;f)+S(r,f).$
\end{lem}

\begin{lem}\label{l4}\cite{6} Let $f$ be a non-constant meromorphic function and let \[R(f)=\frac{\sum\limits _{i=0}^{n} a_{i}f^{i}}{\sum \limits_{j=0}^{m} b_{j}f^{j}}\] be an irreducible rational function in $f$ with constant coefficients $\{a_{i}\}$ and $\{b_{j}\}$ where $a_{n}\not=0$ and $b_{m}\not=0$. Then $$T(r,R(f))=pT(r,f)+S(r,f),$$ where $p=\max\{n,m\}$.\end{lem}
\begin{lem} \label{l5} $N(r,\infty;\frac{M}{f^{d_{M}}}) \leq d_{M}N(r,0;f)+\lambda\overline{N}(r,\infty;f)+S(r,f)$.\end{lem}
\begin{proof}Let $z_{0}$ be a pole of $f$  of order $t$. Then it is a pole of $\frac{M}{f^{d_{M}}}$ of order $n_{1}+2n_{2}+...+kn_{k}=\lambda$.\\
  Let  $z_{0}$ be a zero of $f$  of order $s$. Then it is a pole of $\frac{M}{f^{d_{M}}}$ of order at most $sd_{M}$.\\
So, $N(r,\infty;\frac{M}{f^{d_{M}}}) \leq d_{M}N(r,0;f)+\lambda\overline{N}(r,\infty;f)+S(r,f)$.\end{proof}
\begin{lem}\label{l7} For any two non-constant meromorphic functions $f_{1}$ and $f_{2}$,
\par $N_{p}(r,\infty;f_{1}f_{2}) \leq N_{p}(r,\infty;f_{1})+N_{p}(r,\infty;f_{2})$.\end{lem}
\begin{proof} Let $z_{0}$ be a pole of $f_{i}$ of order $t_{i}$ for $i=1,2.$ Then $z_{0}$ be a pole of $f_{1}f_{2}$ of order at most $t_{1}+t_{2}$.\\
\textbf{Case-1 :} Let $t_{1} \geq p $ and $t_{2} \geq p $. Then $t_{1}+t_{2}\geq p$. So $z_{0}$ is counted at most $p$ times in the left hand side of the above counting function, whereas the same is counted $p+p$ times in the right hand side of the above counting function.\\
\textbf{Case-2 :}  Let $t_{1} \geq p $ and $t_{2} < p $.\\
 \textbf{Subcase-2.1} Let $t_{1}+t_{2}\geq p$. So $z_{0}$ is counted at most $p$ times in the left hand side of the above counting function, whereas the same is counted as $p+\max\{0,t_{2}\}$ times in the right hand side of the above counting function.\\
 \textbf{Subcase-2.2} Let $t_{1}+t_{2}< p$. This case is occurred if $t_{2}$ is negative i.e. if $z_{0}$ is a zero of $f_{2}$. Then $z_{0}$ is counted at most $\max\{0,t_{1}+t_{2}\}$ times whereas the same is counted $p$  times in the right hand side of the above expression. \\
\textbf{Case-3 :} Let $t_{1} < p $ and $t_{2} \geq p $. Then $t_{1}+t_{2}\geq p$. This case can be disposed off as done in Case 2.\\
\textbf{Case-4 :} Let $t_{1} < p $ and $t_{2} < p $ \\
\textbf{Subcase-4.1 :} Let $t_{1}+t_{2}\geq p$. \\ Then $z_{0}$ is counted at most $p$ times whereas the same is counted $\max\{0,t_{1}\}+\max\{0,t_{2}\}$ times in the right hand side of the above expression. \\
\textbf{Subcase-4.2 :} Let $t_{1}+t_{2} < p$. \\ Then $z_{0}$ is counted at most $\max\{0,t_{1}+t_{2}\}$ times whereas $z_{0}$ is counted $\max\{0,t_{1}\}+\max\{0,t_{2}\}$ times in the right hand side of the above counting functions. Combining all the cases, Lemma \ref{l7} follows.\end{proof}
\begin{lem}(\cite{5})\label{l8} $N_{p}(r,0;f^{(k)}) \leq N_{p+k}(r,0;f)+k\overline{N}(r,\infty;f)+S(r,f).$ \end{lem}
\begin{lem}\label{l9} For the differential monomial $M[f]$,\par $N_{p}(r,0;M[f]) \leq d_{M}N_{p+k}(r,0;f)+\lambda\overline{N}(r,\infty;f)+S(r,f)$. \end{lem}
\begin{proof}
Clearly for any non-constant meromorphic function $f$, $N_{p}(r,f) \leq N_{q}(r,f)$ if $p \leq q$.\par
Now by using the above fact and Lemma \ref{l7}, Lemma \ref{l8}, we get
\beas N_{p}(r,0;M[f]) &\leq& \sum\limits_{i=0}^{k} n_{i}N_{p}(r,0;f^{(i)})+S(r,f)\\
&\leq& \sum\limits_{i=0}^{k} n_{i}\{N_{p+i}(r,0;f)+i\overline{N}(r,\infty;f)\}+S(r,f)\\
&\leq& \sum\limits_{i=0}^{k} n_{i}N_{p+i}(r,0;f)+\lambda\overline{N}(r,\infty;f)+S(r,f) \\
&\leq& \sum\limits_{i=0}^{k} n_{i}N_{p+k}(r,0;f)+\lambda\overline{N}(r,\infty;f)+S(r,f) \\
&\leq& d_{M}N_{p+k}(r,0;f)+\lambda\overline{N}(r,\infty;f)+S(r,f).\eeas\end{proof}
\begin{lem}\label{l10}  Let $f$ be a non-constant  meromorphic function and $a(z)$ be a small function in $f$. Let us define $F=\frac{f^{n}}{a}, G=\frac{M}{a}$. Then $FG \not\equiv 1$.\end{lem}
\begin{proof} On contrary assume $FG \equiv 1$. Then in view of Lemma \ref{l5} and the First Fundamental Theorem, we get
 \begin{eqnarray*} (n+d_{M})T(r,f) &=&T(r,\frac{M}{f^{d_{M}}})+S(r,f)\\
 &\leq& d_{M}N(r,0;f)+\lambda\overline{N}(r,\infty;f)+S(r,f)\\ &=&S(r,f),\end{eqnarray*} which is a contradiction.\end{proof}
 \begin{lem}\label{l11}(\cite{1}) Let $F$ and $G$ share $(1,l)$ and $\overline{N}(r,F)=\overline{N}(r,G)$ and $H\not\equiv 0$, where  $F$, $G$ and $H$ are defined as earlier. 
 Then \beas N(r,\infty;H) &\leq& \overline{N}(r,\infty;F)+\overline{N}(r,0;F|\geq 2)+\overline{N}(r,0;G|\geq 2)+\overline{N}_{0}(r,0;F')+\overline{N}_{0}(r,0;G')\\
 &&+\overline{N}_{L}(r,1;F)+\overline{N}_{L}(r,1;G)+S(r,f). \eeas
 \end{lem}
\begin{lem}\label{l12} Let $F$ and $G$ share $(1,l)$.\\ Then $\overline{N}_{L}(r,1;F)\leq \frac{1}{2}\overline{N}(r,\infty;F)+\frac{1}{2}\overline{N}(r,0;F)+S(r,F)$ if $l\geq 1$\\
  and  $\overline{N}_{L}(r,1;F)\leq \overline{N}(r,\infty;F)+\overline{N}(r,0;F)+S(r,F)$ if $l=0$.\end{lem}
\begin{proof}Let $l\geq1$. Then multiplicity of any 1-point of $F$ counted in $ \overline{N}_{L}(r,1;F)$ is at least 3 as $l\geq1$.\\ So, $\overline{N}_{L}(r,1;F) \leq \frac{1}{2}\overline{N}(r,0;F'|F\neq 0) \leq \frac{1}{2}\overline{N}(r,\infty;F)+\frac{1}{2}\overline{N}(r,0;F)+S(r,F)$.\\
 Let $l=0$. Then multiplicity of any 1-point of $F$ counted in $ \overline{N}_{L}(r,1;F)$ is at least 2 as $l=0$.\\
 So, $\overline{N}_{L}(r,1;F) \leq \overline{N}(r,0;F'|F\neq 0) \leq \overline{N}(r,\infty;F)+\overline{N}(r,0;F)+S(r,F)$.\end{proof}
\begin{lem}\label{l13} Let $F$ and $G$ share $(1,l)$ and $H \not\equiv 0$. Then \beas \overline{N}(r,1;F)+ \overline{N}(r,1;G) &\leq& N(r,\infty;H) + \overline{N}^{(2}_{E}(r,1;F)+\overline{N}_{L}(r,1;F)+\overline{N}_{L}(r,1;G)\\ && +\overline{N}(r,1;G)+S(r,f).\eeas \end{lem}
\begin{proof} Clearly, $\overline{N}(r,1;F) = N(r,1;F|=1)+ \overline{N}^{(2}_{E}(r,1;F)+\overline{N}_{L}(r,1;F)+\overline{N}_{L}(r,1;G)$.\\ and by simple calculation, $N(r,1;F|=1)\leq N(r,0;H)+S(r,f)\leq N(r,\infty;H)+S(r,f)$.\end{proof}
\begin{lem}\label{l14}
Let $f$ be a non constant meromorphic function and $a(z)$ be a small function of $f$. Let $F=\frac{f^{n}}{a}$ and $G=\frac{M}{a}$ such that $F$ and $G$ shares $(1,\infty)$. Then one of the following cases holds:
\begin{enumerate}
\item $T(r) \leq N_{2}(r,0;F)+N_{2}(r,0;G)+\ol{N}(r,\infty;F)+\ol{N}(r,\infty;G)$\par
$+\ol{N}_{L}(r,\infty;F)+\ol{N}_{L}(r,\infty;G)+S(r),$
\item  $F\equiv G,$
\item $FG\equiv 1$.
\end{enumerate}
where  $T(r)=\max\{T(r,F),T(r,G)\}$ and $S(r)=o(T(r))$, $r\in I$, $I$ is a set of infinite linear measure of $r\in(0,\infty)$.
\end{lem}
\begin{proof}
 Let $z_{0}$ be a pole of $f$ which is not a pole or zero of $a(z)$. Then $z_{0}$ is a pole of $F$ and $G$ simultaneously. Thus $F$ and $G$ share those pole of $f$ which is not zero or pole of $a(z)$.
Clearly
\beas N(r,H) &\leq& \ol{N}(r,0;F\geq2)+\ol{N}(r,0;G\geq2)+\ol{N}_{L}(r,\infty;F)+\ol{N}_{L}(r,\infty;G)\\
&+& \ol{N}_{0}(r,0;F')+\ol{N}_{0}(r,0;G')+S(r,f)\eeas
Now the proof can be carried out in the line of proof of Lemma 2.13 of \cite{1.1}. So we omit the details.
\end{proof}
\section {Proof of the theorem}
\begin{proof} Let $F=\frac{f^{n}}{a}$ and $G=\frac{M[f]}{a}$. Then $F-1=\frac{f^{n}-a}{a}$, $G-1=\frac{M[f]-a}{a}$. Since $f^{n}$ and $M[f]$ share $(a,l)$, it follows that $F$ and $G$ share $(1,l)$ except the zeros and poles of $a(z)$. Now we consider the following cases.\\
{\bf Case 1} Let $H\not\equiv 0$.\\
\textbf{Subcase-1.1.} $l\geq 1$\\
Using the Second Fundamental Theorem and Lemmas \ref{l13}, \ref{l11} we get\par
\bea\nonumber T(r,F)+T(r,G) &\leq& \overline{N}(r,\infty;F)+\overline{N}(r,\infty;G)+\overline{N}(r,0;F)+\overline{N}(r,0;G)+N(r,H) \\
&& \nonumber+ \overline{N}^{(2}_{E}(r,1;F)+\overline{N}_{L}(r,1;F)+\overline{N}_{L}(r,1;G)+\overline{N}(r,1;G)\\
&& \nonumber
-\overline{N}_{0}(r,0;F^{'})-\overline{N}_{0}(r,0;G^{'})+S(r,f)\\
& \label{t2} \leq& 2\overline{N}(r,\infty;F)+\overline{N}(r,\infty;G)+N_{2}(r,0;F)+N_{2}(r,0;G) +\overline{N}^{(2}_{E}(r,1;F)\\
&& \nonumber +2\overline{N}_{L}(r,1;F)+2\overline{N}_{L}(r,1;G)+\overline{N}(r,1;G) +S(r,f).\eea
\textbf{Subsubcase-1.1.1.} For $l=1$\\
From inequality (\ref{t2}) and in view of Lemmas \ref{l12}, \ref{l9} we get
\beas T(r,F)+T(r,G) & \leq& 2\overline{N}(r,\infty;F)+\overline{N}(r,\infty;G)+N_{2}(r,0;F)+N_{2}(r,0;G) +\overline{N}^{(2}_{E}(r,1;F)\\
&& +2\overline{N}_{L}(r,1;F)+2\overline{N}_{L}(r,1;G)+\overline{N}(r,1;G)+S(r,f)\\
& \leq& \frac{5}{2}\overline{N}(r,\infty;F)+\overline{N}(r,\infty;G)+\frac{1}{2}\overline{N}(r,0;F)+\mu_{2}N_{\mu_{2}^{*}}(r,0;f)+N_{2}(r,0;G)\\
&&+\overline{N}^{(2}_{E}(r,1;F)+\overline{N}_{L}(r,1;F)+2\overline{N}_{L}(r,1;G)+\overline{N}(r,1;G)+S(r,f)\\
& \leq& \frac{5}{2}\overline{N}(r,\infty;F)+\overline{N}(r,\infty;G)+\frac{1}{2}\overline{N}(r,0;F))+\mu_{2}N_{\mu_{2}^{*}}(r,0;f)+N_{2}(r,0;G)\\
&&+N(r,1;G)+S(r,f).\eeas
i.e., for any $\varepsilon > 0$
\beas nT(r,f) &\leq& (\lambda+\frac{7}{2})\overline{N}(r,\infty;f)+\frac{1}{2}\overline{N}(r,0;f)+\mu_{2}N_{\mu_{2}^{*}}(r,0;f)+d_{M}N_{2+k}(r,0;f)+S(r,f)\\
& \leq& \{(\lambda+\frac{7}{2})-(\lambda+\frac{7}{2})\Theta(\infty,f)+\frac{1}{2}-\frac{1}{2}\Theta(0,f)+\mu_{2}-\mu_{2}\delta_{\mu_{2}^{*}}(0,f)\\
&&+d_{M}-d_{M}\delta_{2+k}(0,f)+\varepsilon\}T(r,f)+S(r,f).\eeas\\
i.e., $\{(\lambda+\frac{7}{2})\Theta(\infty,f)+\frac{1}{2}\Theta(0,f)+\mu_{2}\delta_{\mu_{2}^{*}}(0,f)+d_{M}\delta_{2+k}(0,f)-\varepsilon\}T(r,f)\\ \leq (\Gamma_{M}+\mu_{2}+4-n)T(r,f)+S(r,f)$, which is a contradiction.
\\
\textbf{Subsubcase-1.1.2.} For $l\geq 2$\\
Now by using the inequality (\ref{t2}) and Lemma \ref{l9}, we get
\beas T(r,F)+T(r,G) & \leq& 2\overline{N}(r,\infty;F)+\overline{N}(r,\infty;G)+N_{2}(r,0;F)+N_{2}(r,0;G) +\overline{N}^{(2}_{E}(r,1;F)
\\
&&+2\overline{N}_{L}(r,1;F)+2\overline{N}_{L}(r,1;G)+\overline{N}(r,1;G)+S(r,f)\\
& \leq& 2\overline{N}(r,\infty;F)+\overline{N}(r,\infty;G)+\mu_{2}N_{\mu_{2}^{*}}(r,0;f)+N_{2}(r,0;G)+ N(r,1;G)\\
&& +S(r,f).\eeas
i.e., for any $\varepsilon > 0$
\beas nT(r,f) &\leq & (\lambda+3)\overline{N}(r,\infty;f)+\mu_{2}N_{\mu_{2}^{*}}(r,0;f)+d_{M}N_{2+k}(r,0;f)+S(r,f)\\
& \leq & \{(\lambda+3)-(\lambda+3)\Theta(\infty,f)+\mu_{2}-\mu_{2}\delta_{\mu_{2}^{*}}(0,f)\\
&& +d_{M}-d_{M}\delta_{2+k}(0,f)+\varepsilon\}T(r,f)+S(r,f).\eeas
i.e., $\{(\lambda+3)\Theta(\infty,f)+\mu_{2}\delta_{\mu_{2}^{*}}(0,f)+d_{M}\delta_{2+k}(0,f)-\varepsilon\}T(r,f)\\
\leq (\Gamma_{M}+3+\mu_{2}-n)T(r,f)+S(r,f)$, which is a contradiction.\\
\textbf{Subcase-1.2.}
$l=0$\\
Then by using the Second Fundamental Theorem and Lemma \ref{l13}, \ref{l11}, \ref{l12}, \ref{l9} we get\\
\bea \nonumber T(r,F)+T(r,G) \nonumber&\leq& \overline{N}(r,\infty;F)+\overline{N}(r,0;F)+\overline{N}(r,1;F)+\overline{N}(r,\infty;G) +\overline{N}(r,0;G)\\
\nonumber && +\overline{N}(r,1;G)-\overline{N}_{0}(r,0;F')-\overline{N}_{0}(r,0;G')+S(r,F)+S(r,G)\\
\nonumber & \leq & \overline{N}(r,\infty;F)+\overline{N}(r,0;F)+\overline{N}(r,\infty;G)+\overline{N}(r,0;G)+N(r,\infty;H)\\
\nonumber && +\overline{N}^{(2}_{E}(r,1;F)+\overline{N}_{L}(r,1;F)+\overline{N}_{L}(r,1;G)+\overline{N}(r,1;G)\\
\nonumber && -\overline{N}_{0}(r,0;F')-\overline{N}_{0}(r,0;G')+S(r,F)+S(r,G)\\
\nonumber & \leq& 2\overline{N}(r,\infty;F)+\overline{N}(r,\infty;G)+N_{2}(r,0;F)+N_{2}(r,0;G)+\overline{N}^{(2}_{E}(r,1;F)\\
\nonumber && +2\overline{N}_{L}(r,1;F)+2\overline{N}_{L}(r,1;G)+\overline{N}(r,1;G)+S(r,f) \\
\nonumber & \leq & 2\overline{N}(r,\infty;F)+\overline{N}(r,\infty;G)+\mu_{2}N_{\mu_{2}^{*}}(r,0,f)+N_{2}(r,0;G)\\
\nonumber && +2(\overline{N}(r,\infty;F)+\overline{N}(r,0;F))+\overline{N}(r,\infty;G)+\overline{N}(r,0;G)\\
\nonumber && +\overline{N}^{(2}_{E}(r,1;F)+\overline{N}_{L}(r,1;G)+\overline{N}(r,1;G)+S(r,f)\\
&\label{r11} \leq&  4\overline{N}(r,\infty;F)+\mu_{2}N_{\mu_{2}^{*}}(r,0,f)+N_{2}(r,0;G)+2\overline{N}(r,\infty;G)\\
\nonumber && +\overline{N}(r,0;G)+2\overline{N}(r,0;F)
+T(r,G)+S(r,f)\eea
i.e., for any $\varepsilon > 0$
\beas nT(r,f) &\leq& (2\lambda+6)\overline{N}(r,\infty;f)+2\overline{N}(r,0;f)+\mu_{2}N_{\mu_{2}^{*}}(r,0,f)+d_{M}N_{1+k}(r,0;f)\\
\nonumber && +d_{M}N_{2+k}(r,0;f)+S(r,f)\\
& \leq& \{(2\lambda+6)-(2\lambda+6)\Theta(\infty,f)+2-2\Theta(0,f)+\mu_{2}-\mu_{2}\delta_{\mu_{2}^{*}}(0,f)\\
&& +2d_{M}-d_{M}\delta_{1+k}(0,f)-d_{M}\delta_{2+k}(0,f)+\varepsilon\}T(r,f)+S(r,f).\eeas\\
i.e., \beas &&\{(2\lambda+6)\Theta(\infty,f)+2\Theta(0,f)+\mu_{2}\delta_{\mu_{2}^{*}}(0,f)+d_{M}\delta_{1+k}(0,f)+d_{M}\delta_{2+k}(0,f)-\varepsilon\}T(r,f)\\ &\leq& (2\Gamma_{M}+8+\mu_{2}-n)T(r,f)+S(r,f),\eeas which is a contradiction.\\

{\bf Case 2.} Let $H\equiv 0$.\\
On Integration we get,
\begin{equation}\label{3.3} \frac{1}{G-1}\equiv\frac{A}{F-1}+B,\end{equation}
 where $A(\neq 0), B$ are complex constants.
Then $F$ and $G$ share $(1,\infty)$. Also by construction of $F$ and $G$ we see that $F$ and $G$ share $(\infty,0)$ also. \par
So using Lemma \ref{l9} and condition (\ref{e1.12}), we obtain
\beas && N_{2}(r,0;F)+N_{2}(r,0;G)+\ol{N}(r,\infty;F)+\ol{N}(r,\infty;G)+\ol{N}_{L}(r,\infty;F)+\ol{N}_{L}(r,\infty;G)+S(r)\\
 &\leq& \mu_{2}N_{\mu_{2}^{*}}(r,0;f)+d_{M}N_{2+k}(r,0;f)+(\lambda+3)\overline{N}(r,\infty;f)+S(r)\\
 &\leq&\{(3+\lambda+d_{M}+\mu_{2})-((\lambda+3)\Theta(\infty,f)+\delta_{\mu_{2}^{*}}(0,f)+d_{M}\delta_{2+k}(0,f))\}T(r,f)+S(r)\\
 &<& T(r,F)+S(r)\eeas
Hence inequality (1) of Lemma \ref{l14} does not hold. Again in view of Lemma \ref{l10}, we get  $F\equiv G$, i.e., $f^{n} \equiv M[f]$.
\end{proof}
\begin{center} {\bf Acknowledgement} \end{center}
This research work is supported by the Council Of Scientific and Industrial Research, Extramural
Research Division, CSIR Complex, Pusa, New Delhi-110012, India, under the sanction project no. 25(0229)/14/EMR-II.
The authors also wish to thank the referee for his/her valuable remarks and suggestions towards the improvement of the paper.

\end{document}